\documentclass[12pt]{amsart}
\usepackage{amsmath,amsfonts,amsthm,amssymb,mathrsfs,amsxtra,amscd,latexsym, xcolor}
\usepackage[hmargin=2cm,vmargin=3cm]{geometry}
\usepackage{hyperref}

\usepackage{amsmath,amsthm,amssymb,enumerate,amsfonts,mathrsfs,amsxtra,amscd,latexsym}

\newtheorem{theorem}{Theorem}[section]

\newtheorem{lemma}[theorem]{Lemma}

\newtheorem{example}{Example}

\newtheorem*{remark}{Remark}

\numberwithin{equation}{section}

\newcommand{\sm}{\left(\begin{smallmatrix}}
\newcommand{\esm}{\end{smallmatrix}\right)}
\newcommand{\mat}{\left(\begin{matrix}}
\newcommand{\emat}{\end{matrix}\right)}

\def\SL{\mathrm{SL}}

  \def\Z{\mathbb Z}

\newcommand{\beq}{\begin{eqnarray*}}
\newcommand{\eeq}{\end{eqnarray*}}
\newcommand{\beqn}{\begin{eqnarray}}
\newcommand{\eeqn}{\end{eqnarray}}

\newcommand{\ben}{\begin{enumerate}}
\newcommand{\een}{\end{enumerate}}

\begin{document}

\title{The Riemann hypothesis for period polynomials of modular forms}

\author{}

\address{}
\email{}

\author{Seokho Jin}

\address{School of Mathematics, Korea Institute for Advanced Study, Hoegiro 85, Dongdaemun-gu, Seoul 130-722, Korea}
\email{seokhojin@kias.re.kr}

\author{Wenjun Ma}

\address{School of mathematics, Shandong University, Jinan, Shandong, 250100, China}
\email{wenjunma.sdu@hotmail.com}

\author{Ken Ono}

\address{Department of Mathematics and Computer Science, Emory University, Atlanta, Georgia 30322}
\email{ono@mathcs.emory.edu}

\author{Kannan Soundararajan}
\address{Department of Mathematics, Stanford University, Stanford, California 94305}
\email{ksound@math.stanford.edu}

\dedicatory{For Don Zagier in celebration of his 65th birthday}

\begin{abstract}  The period polynomial $r_f(z)$ for an even weight $k\geq 4$ newform $f\in S_k(\Gamma_0(N))$  is the generating function for the critical values of $L(f,s)$.
It has a functional equation relating $r_f(z)$
to $r_f\left(-\frac{1}{Nz}\right)$.
We prove the Riemann Hypothesis for these polynomials: that  the zeros
of $r_f(z)$ lie on the circle  $|z|=1/\sqrt{N}$. 
 We prove that these zeros are equidistributed when either $k$ or $N$ is large.
\end{abstract}

\thanks{The first author thanks KIAS for its generous support. The second author thanks the China Scholarship Council for its generous support. The third author thanks the support of the Asa Griggs Candler Fund and the NSF. The fourth author thanks the support of the NSF, and the Simons Foundation for a Simons Investigator Grant. The authors thank YoungJu Choie, Yuri Manin, Ram Murty, Ken Ribet, Drew Sutherland, and Don Zagier for useful comments and discussions.
}

\maketitle


\section{Introduction} \label{section1}

\noindent Let $f \in S_k(\Gamma_0(N))$ be a newform \cite{AL, Li}  of even weight $k$ and level $N$.  Associated to $f$ is its $L$-function $L(f,s)$, which has been normalized so that the completed $L$-function,
$$
\Lambda(f,s) := \Big(\frac{\sqrt{N}}{2\pi}\Big)^s \Gamma(s) L(f,s),   
$$ 
satisfies the functional equation $\Lambda(f,s) = \epsilon(f) \Lambda(f,k-s)$, with $\epsilon(f)= \pm 1$.  Recall that the completed $L$-function arises as a period integral of the newform $f$: 
\begin{equation} 
\label{1.1} 
\Lambda(f,s) = N^{s/2} \int_0^{\infty} f(iy) y^{s} \frac{dy}{y}.
\end{equation} 
The focus of this paper is the {\sl period polynomial} associated to $f$, the
degree $k-2$ polynomial
\begin{equation}\label{PolyDef}
r_f(z):=\int_0^{i\infty}f(\tau)(\tau-z)^{k-2}d\tau.
\end{equation}
Expanding $(\tau-z)^{k-2}$, and using \eqref{1.1}, we may also express the period polynomial by 
\begin{equation} 
\label{1.3} 
r_f(z) = i^{k-1} N^{-\frac{k-1}{2}}  \sum_{n=0}^{k-2} \binom{k-2}{n} (\sqrt{N}i z)^n \Lambda(f,k-1-n), 
\end{equation}
or equivalently as 
\begin{equation} 
\label{1.4} 
r_f(z) = -\frac{(k-2)!}{(2\pi i)^{k-1} } \sum_{n=0}^{k-2} \frac{(2\pi iz)^n}{n!} L(f,k-n-1). 
\end{equation} 
In other words, $r_f(z)$ is a generating function for the critical values $L(f,1)$, $L(f,2)$, $\dots$, $L(f,k-1)$. For general facts on period polynomials,
the reader is encouraged to see \cite{CPZ, Kn, KZ, PP, Z}; other papers broadly related to the themes of this paper are \cite{GMR, MSW}.

Using the functional equation $\Lambda(f,s) = \epsilon(f) \Lambda(f,k-s)$ in \eqref{1.3}, we find that 
$$ 
r_f(z)  = - i^{k} \epsilon(f) (\sqrt{N}z)^{\frac{k-2}{2}} r_f\Big(-\frac{1}{Nz}\Big), 
$$ 
so that if $\rho$ is a zero of $r_f(z)$ then so is $-1/(N\rho)$.  In analogy with the Riemann hypothesis, we may ask whether all the zeros of $r_f(z)$ lie on the circle $|\rho|=1/\sqrt{N}$.   For Hecke eigenforms on $\SL_2(\Z)$, this was recently established by El-Guindy and Raji \cite{ER}, who showed that the zeros of $r_f(z)$ (for $N=1$) are all on the unit circle $|z|=1$. Their work was inspired by the previous work by Conrey, Farmer and Imamo$\mathrm{\bar{g}}$lu \cite{CFI}, who proved an analogous result for odd period polynomials again for full level.  
We show that this ``Riemann hypothesis" holds in general for all newforms of 
weight at least $4$ and any level.

\begin{theorem}\label{NHecke} For any even integer $k$ at least $4$, and any level $N$, all of the zeros of the period polynomial $r_f(z)$ are on the circle $|z| =1/\sqrt{N}$.  
\end{theorem}

\begin{remark}
Period polynomials for weight 2 newforms $f$ are constant  multiples of  $L(f,1)$.
\end{remark}


\begin{example}
The period polynomial for the normalized Hecke eigenform $\Delta(z)\in S_{12}(\Gamma_0(1))$ is
\begin{displaymath}
\begin{split}
r_\Delta(z)=\omega_\Delta^+r_\Delta^+(z)+&\omega_\Delta^-r_\Delta^-(z)
\approx 0.114379i\left(\frac{36}{691}z^{10}-z^8+3z^6-3z^4+z^2-\frac{36}{691}\right)\\
&\ \ \ \ \ \ \ \ \ \ \ \ \ \ \ \ \ \ \ \ \ \ \ \ +0.00926927 (4z^9-25z^7+42z^5-25z^3+4z).
\end{split}
\end{displaymath}
All ten zeros of $r_{\Delta}(z)$ are on $|z|=1$.
\end{example}

\begin{example}
For the unique weight $4$ newform $f(z)=q-4q^3-2q^5+\cdots$ on $\Gamma_0(8)$, we have
\beq
L(f,1)\approx 0.3545006\dots, \ \ \
L(f,2)\approx 0.6900311\dots, \ \ \
L(f,3)\approx 0.8746953\dots,
\eeq
which in turn implies that
$r_{f}(z)\approx0.0564205361iz^2 + 0.0349573870z  - 0.00705256701815496i.$
The roots are $\approx \pm 0.17037672 + 0.30979311 i$, and their norms are $\approx 1/(2\sqrt{2})$.
\end{example}

\begin{remark} Manin \cite{M} has used the work of Conrey, Farmer and Imamo$\mathrm{\bar{g}}$lu \cite{CFI}  to construct zeta functions which satisfy the Riemann Hypothesis. He suggests that these polynomials  arise from non-Tate motives and geometric objects lying below ${\text {\rm Spec}} {\bf\ Z}$ but not over ${\bf F}_1$. Using the $P_f(z)$ defined below, one obtains further such polynomials {\it mutatis mutandis.}
\end{remark}

If the weight or level is large enough, then the zeros of $r_f$ are regularly spaced on the circle 
$|z|=1/\sqrt{N}$.  To state this conveniently, and for our later work, we shall put $m:=(k-2)/2$ throughout and 
define 
\begin{equation} 
\label{1.5} 
P_f(z) = \tfrac 12 \binom{2m}{m} \Lambda(f,\tfrac k2) + \sum_{j=1}^{m} \binom{2m}{m+j} \Lambda(f,\tfrac k2+j) z^j.
\end{equation} 
Then, using the functional equation, we see that 
\begin{equation} 
\label{1.6}
r_f\Big(\frac{z}{i\sqrt{N}}\Big) = i^{k-1} N^{-\frac{k-1}{2}} \epsilon(f) z^{m} \Big( P_f(z) + \epsilon(f) P_f\Big(\frac 1z\Big)\Big). 
\end{equation} 
Therefore, to understand the zeros of $r_f$, it is enough to understand the zeros of $P_f(z)+\epsilon(f) P_f(1/z)$, and Theorem \ref{NHecke} states that this function has all its zeros on the unit circle $|z|=1$.  If we restrict to the 
unit circle $|z|=1$, then $P_f(z) + \epsilon(f) P_f(1/z)$ is either a trigonometric cosine or a trigonometric sine polynomial (depending on whether $\epsilon(f)$ equals $1$ or $-1$), and our proof of Theorem \ref{NHecke} proceeds by finding the right number of sign changes as $z$ varies over the unit circle.  If $k$ or $N$ is large enough, the proof allows us to establish the following result on the location of the roots.

 \begin{theorem}\label{equi} 
 The following are true.
 \begin{itemize}
 \item[(i)] Suppose that $k=4$. If $\epsilon(f)=-1$, then the zeros of $r_f(z)$ are  $\pm i/\sqrt{N}.$ If $\epsilon(f)=1$ and $N$ is sufficiently large, then the zeros of $r_f(z)$
 are located at $\pm (1+O(N^{-\frac 14+\epsilon}))/\sqrt{N}$.   

\item[(ii)] If $k\ge 6$ and either $N$ or $k$ is large enough, then the roots of $r_f(z)$ may be written as 
$$
\frac{1}{i\sqrt{N}} \exp\Big( i \theta_\ell + O\Big(\frac{1}{2^k\sqrt{N}}\Big) \Big),
$$ 
where for $0\le \ell \le 2m-1$ we denote by $\theta_\ell$ the unique solution in $[0,2\pi)$ to the equation  
$$ 
m\theta_\ell - \frac{2\pi}{\sqrt{N}} \sin \theta_\ell = \begin{cases} 
\frac \pi2 +\ell \pi &\text{if  } \epsilon(f)=1\\ 
\ell \pi &\text{if  } \epsilon(f)=-1. \\
\end{cases}
$$
\end{itemize}
\end{theorem}

Our arguments readily allow us to quantify the results in Theorem 1.2.  For example, the arguments in Section 6 
give that in part (ii) above, the implied $O$-constant may be taken as $10^9$, although this is a gross overestimate.  
The arguments in Section 5 locate sign changes even if the values of $k$ or $N$ are only moderately large.

Suppose that $\epsilon(f)=1$.  By counting sign changes, one consequence of Theorem \ref{1.1} is that $P_f(-1)$ has sign $(-1)^{m}$.  
In other words, if $\epsilon(f)=1$, then we must have 
\begin{equation} 
\label{1.7} 
 \frac 12 \binom{2m}{m} (-1)^m \Lambda(f,\tfrac k2) + \sum_{j=0}^{m-1} (-1)^j \binom{2m}{2m-j} \Lambda(f,k-1-j)    >0.  
\end{equation} 
For any weight $k$, this inequality is clear for large enough $N$ since the term $j=0$ above dominates all other terms.  
However it is interesting that such an inequality holds for all small weights and small level as well, and we wonder if it has any other 
significance.  In Section 4 we give a proof of this inequality in the weight $6$ case based on the Hadamard factorization formula.  We also 
give there a more illuminating proof of this inequality based on the Riemann hypothesis for $\Lambda(f,s)$.

\section{Preliminaries}\label{pre}

\noindent Here we collect  preliminary facts about $L$-functions which we shall 
find useful.  The completed $L$-function $\Lambda(f,s)$ is an entire function of order $1$.  Its zeros all 
lie in the strip $|\text{Re}(s) -\tfrac k2|<\tfrac 12$, with the Riemann hypothesis predicting that all zeros lie 
on the line Re$(s) =\tfrac k2$.  Recall also that the central value $\Lambda(f,\tfrac k2)$ is known to be 
non-negative by the work of Waldspurger \cite{W}.

Hadamard's factorization formula applies to the entire function $\Lambda(f,s)$, and we may write 
\begin{equation} 
\label{2.1} 
\Lambda(f,s) = e^{A+Bs} \prod_{\rho} \Big(1-\frac s{\rho}\Big) e^{s/\rho}. 
\end{equation} 
Here the product is over all the zeros of $\Lambda(f,s)$ (that is, the non-trivial zeros of $L(f,s)$), and $A$ and $B$ are constants.  Note that if $\rho$ is a zero then so too are $\overline{\rho}$ and $k-\rho$.   Since $\Lambda(f,s)$ is real-valued on the real line, and in view of the functional equation, we have that $B$ is real-valued and 
$$ 
B  = - \sum_{\rho} \text{Re} \frac{1}{\rho} = -\sum_{\rho} \frac{\text{Re }(\rho)}{|\rho|^2}. 
$$ 
These considerations also show that for real $s$
\begin{equation} 
\label{2.2} 
\Lambda(f,s) = e^A \prod_{\rho \in {\Bbb R}} \Big( 1-\frac{s}{\rho}\Big) \prod_{\text{Im}(\rho) >0} \Big|1-\frac{s}{\rho}\Big|^2,  
\end{equation} 
where we have paired the complex conjugate roots together so that the product is convergent.  

\begin{lemma}  \label{lem2.1} The function $\Lambda(f,s)$ is monotone increasing for $s\ge \tfrac k2 +\tfrac 12$.  
Moreover, we have
$$ 
0\le \Lambda(f,\tfrac k2) \le \Lambda(f,\tfrac k2+1) \le \Lambda(f,\tfrac k2+2 )\le \ldots .
$$ 
If $\epsilon(f)$ is $-1$, then $\Lambda(f,\tfrac k2) =0$ and 
$$ 
0\leq \Lambda(f,\tfrac k2+1) \le \frac 12 \Lambda(f,\tfrac k2 +2)\le \frac 13 \Lambda(f,\tfrac k2 +3) \le \ldots. 
$$
\end{lemma} 

Monotonicity results such as Lemma \ref{lem2.1} are familiar in the literature; for example, the Riemann hypothesis for $L$-functions 
is equivalent to the monotonicity of the absolute value of the completed $L$-function along horizontal lines starting from the critical line.  
In a different context Stark and Zagier observed a similar result  in \cite{SZ}.
\begin{proof} Since all the zeros lie in $|\text{Re}(s)-\frac k2 |<\frac 12$, we see that $|1-s/\rho|$ is increasing 
for $s\ge \frac k2 +\frac 12$.  So by \eqref{2.2} it follows that $\Lambda(f,s)$ is increasing in Re$(s)\ge \frac k2 +\frac 12$.  Further, we have
$$ 
\Big| 1- \frac{k/2}{\rho} \Big| \le \Big| 1- \frac{k/2+1}{\rho}\Big|, 
$$ 
and so $\Lambda(f,k/2) \le \Lambda(f,k/2+1)$.  When $\epsilon(f)=-1$, we apply the same reasoning and now take into account that there must be a zero of odd order at $\frac k2$.  
\end{proof}  

We record a useful inequality for $L$-values in the range of absolute convergence. 

\begin{lemma}  \label{lem2.2}  If $0 < a < b$ and $k$ is the weight of $f$, then we have 
$$ 
\frac{L(f,\frac{k+1}{2} +a)}{L(f,\frac{k+1}{2} +b)} \le \frac{\zeta(1+a)^2}{\zeta(1+b)^2}. 
$$ 
\end{lemma} 
\begin{proof}  The Euler product for $L(f,s)$ gives rise to 
$$ 
-\frac{L^{\prime}}{L}(f,s)  = \sum_{n=1}^{\infty} \frac{\Lambda_f(n)}{n^s}, 
$$  
where $|\Lambda_f(n)| \le 2 n^{\frac{k-1}{2}}\Lambda(n)$ for all $n$.  Here $\Lambda(n)$ is the usual von Mangoldt function, and 
this estimate is an alternative way of encoding the Ramanujan bounds established by Deligne \cite{Del} (see also 
Li \cite{Li} for the Euler factors at primes dividing the level).   The point is that  for prime powers $n=p^r$ we have $\Lambda_f(n) = (\alpha_p^r + \beta_p^{r}) \log p$,
where the $p$th Fourier coefficient of $f$ satisfies $a(p) = \alpha_p + \beta_p$.

Therefore, we have
$$
\frac{L(f,\frac{k+1}{2}+a)}{L(f,\frac{k+1}{2}+b)} = \exp\Big(\int_a^b -\frac{L^{\prime}}{L}(f,\tfrac{k+1}{2} +t) dt \Big) 
\le \exp\Big( 2 \int_a^b -\frac{\zeta^{\prime}}{\zeta}(1+t) dt \Big) = \frac{\zeta(1+a)^2}{\zeta(1+b)^2}. 
$$ 
\end{proof}

\section{The weight $4$ case}  

\noindent If $f$ is a form of weight $k=4$ (so $m=(k-2)/2=1$), then 
$P_f(z) = \Lambda(f,2) + \Lambda(f,3)z$.
If $\epsilon(f)=-1$, then the roots of $P_f(z)-P_f(1/z)=\Lambda(f,3)(z-1/z)$ are at 
$z=\pm 1$ and so the period polynomial has roots at $\pm i/\sqrt{N}$.  

If $\epsilon(f)=1$, then with $z=e^{i\theta}$ we have
$P_f(z) + P_f(1/z) = 2 \Lambda(f,2) + 2 \Lambda(f,3) \cos \theta. 
$ 
Since $\Lambda(f,2) < \Lambda(f,3)$ by Lemma \ref{lem2.1}, the above equation has two solutions 
for $\theta \in [0,2\pi)$: namely, $\theta$ satisfying $\cos \theta = - \Lambda(f,2)/\Lambda(f,3)$. This completes the proof of Theorem \ref{NHecke} for weight $4$.   

 Note that $\Lambda(f,3) \gg N^{\frac 32}$ for large $N$, while the Phr{\' a}gmen-Lindel{\" o}f principle  gives
 $\Lambda(f,2) \le \max_{t\in {\Bbb R}} |\Lambda(f,\tfrac 52+\epsilon +it)| \ll N^{\frac 54+\epsilon}$ (this is the ``convexity bound" for $L$-functions).  
 Therefore  the ratio $\Lambda(f,2)/\Lambda(f,3)$ is small (precisely $\ll N^{-\frac 14+ \epsilon}$), and hence the corresponding values of $\theta$ tend to $\pi/2$ and $3\pi/2$.  Thus for large level, the zeros of the period polynomial (in the $\epsilon(f)=1$ case) are located at $\pm (1+O(N^{-\frac 14+\epsilon}))/\sqrt{N}$.  

\section{The weight $6$ case}  

\noindent If $f$ is a form of weight $k=6$ (so that $m=2$) then 
$$ 
P_f(z) = 3\Lambda(f,3) + 4 \Lambda(f,4) z + \Lambda(f,5) z^2. 
$$ 
If $\epsilon(f)=-1$, then we are interested in the roots of 
$$ 
P_f(z) - P_f(1/z) = \Big(z-\frac 1z\Big) \Big( 4\Lambda(f,4) + \Lambda(f,5) \Big(z +\frac{1}{z}\Big)\Big). 
$$  
Clearly there are two solutions $z=\pm 1$.  Since $\epsilon(f)=-1$, we know that 
$2\Lambda(f,4) < \Lambda(f,5)$ by Lemma \ref{2.1}, and so there are two solutions in $[0,2\pi)$ to $\cos \theta =
- 2\Lambda(f,4)/\Lambda(f,5)$.  Thus we have shown Theorem \ref{NHecke} in this case.  Moreover if $N$ is large, then $\Lambda(f,4)/\Lambda(f,5)$ is small and $\theta$ tends to $\pi/2$ or $3\pi/2$.   So for large $N$ (and odd sign)  the period polynomial has two zeros exactly at $\pm i/\sqrt{N}$ and the other two zeros are very near $\pm 1/\sqrt{N}$.  

It remains now to consider when $\epsilon(f)=1$.  With $z=e^{i\theta}$ we must show that 
 \begin{equation} 
 \label{4.1} 
 P_f(z) + P_f(1/z) = 2\cos (2\theta) \Lambda(f,5) + 8\cos \theta \Lambda(f,4) + 6\Lambda(f,3)  
 \end{equation} 
 has two zeros in $[0,\pi]$ (and therefore four zeros in $[0,2\pi)$).  
Differentiating the above with respect to $\theta$ gives $-8\sin \theta ( \Lambda(f,4)+ \cos \theta \Lambda(5))$ so that there are critical points at $\theta =0$, $\pi$, and at the solution $\theta_0 \in (0,\pi)$ to $\cos \theta =-\Lambda(f,4)/\Lambda(f,5)$.   We would like the quantity in \eqref{4.1} to be positive at $\theta=0$, negative at $\theta_0$ and positive again at $\theta=\pi$,  which would ensure two zeros in $(0,\pi)$ (and note that these conditions are also necessary for the period polynomial to have all zeros on a circle).

 The value at $\theta= 0$ is clearly positive.  That the value should be positive at $\pi$ is equivalent to 
 \begin{equation} 
 \label{4.2} 
 \Lambda(f,5) + 3\Lambda(f,3) > 4 \Lambda(f,4). 
 \end{equation}  
 The condition that the value should be negative at $\theta_0$  is equivalent to 
 \begin{equation}
 \label{4.3} 
 \Lambda(f,5)^2 + 2 \Lambda(f,4)^2 \ge 3\Lambda(f,3)\Lambda(f,5). 
 \end{equation} 
 
 \begin{lemma}  Suppose $a_1$, $a_2$, $b_1$, $b_2$, and $c_1$, $c_2$ are all positive with $a_i \ge \max(b_i,c_i)$.   Suppose 
 that $ a_i +\gamma c_i \ge (1+\gamma) b_i$, where  $\gamma$ is positive. 
Then $ a_1 a_2 + \gamma c_1 c_2 \ge (1+\gamma) b_1 b_2$.  
\end{lemma} 
\begin{proof} Multiply the relation $ a_1 + \gamma c_1 \ge (1+\gamma) b_1$ by $b_2$.  It suffices to show that 
$$ 
 a_1 a_2 + \gamma c_1 c_2 \ge  a_1 b_2 + \gamma c_1 b_2; 
$$ 
or, rearranging that 
$
a_1 (a_2 -b_2) \ge \gamma c_1(b_2 -c_2). 
$
Since $(a_2 -b_2)\geq 0$, and $a_1 \ge c_1$, the LHS above is at least $ c_1 (a_2 -b_2)$ which is 
$\ge \gamma c_1(b_2-c_2)$.  
\end{proof} 
 
 \begin{proof}[Proof of \eqref{4.2}]   We use Lemma 4.1 suitably, together with the Hadamard factorization formula (\eqref{2.1} and \eqref{2.2}), proceeding zero by zero.     
 We use the Hadamard formula for $\Lambda(f,3)$, $\Lambda(f,4)$ and $\Lambda(f,5)$; note that at all these values $\Lambda$ is known to be non-negative (This is clear for $4$ and $5$, and work of Waldspurger for $3$.), so we can also assume that the products are taken with absolute values.
   
 Suppose first that $\rho=3+z$ is a real zero, and then $6-\rho=3-z$ is also a real zero.  
 (Note that even if $\rho =3$, we get zeros of even multiplicity at the center, which may be paired.)   Then note that 
 this pair of zeros contributes to $\Lambda(f,5)$ the amount $a=(4-z^2)/(9-z^2)$,  to $\Lambda(f,4)$ the amount $b=(1-z^2)/(9-z^2)$, and to $\Lambda(f,3)$ the amount $c=z^2/(9-z^2)$  (using here the absolute value remark).  Note that with $\gamma =3$ we have the inequality $a+3c \ge 4b$.  
 
 Now consider a zero $\rho = 3+ iy$ on the critical line, and pair it with its conjugate $3-iy$.  These 
 contribute to $\Lambda(f,5)$ the amount $a= (4+y^2)/(9+y^2)$, to $\Lambda(f,4)$ the amount $b=(1+y^2)/(9+y^2)$ and 
 to $\Lambda(f,3)$ the amount $c=y^2/(9+y^2)$, and we check again that $a+ 3c\ge 4b$ (and indeed equality holds).  
 
 Finally consider a zero $\rho = 3+z$ not on the critical line with $z= x+iy$. This comes in a set of four zeros $3 \pm x \pm iy$.  
 Note that these four zeros contribute (multiply through by $|\rho|^2|6-\rho|^2$) to $\Lambda(f,5)$ an amount $a= |4-z^2|^2$, to $\Lambda(f,4)$ an amount $b=|1-z^2|^2$, and to $\Lambda(f,3)$ the amount $c=|z^2|^2$.  We can check again that $a+3c \ge 4b$.  
 
 Thus, when grouped as above, each group of zeros appearing in the Hadamard formula satisfies a version of \eqref{4.2}.  
 By Lemma 4.1, taking products of these groups of zeros we again obtain a version of \eqref{4.2}.   Letting these products 
 run over all zeros and taking the limit, we obtain \eqref{4.2}.
 \end{proof} 
 
 \begin{proof}[Proof of \eqref{4.3}]  This proof is similar, appealing to Lemma 4.1 with $\gamma =2$ and using Hadamard's formula 
 and grouping zeros as before.  
 \end{proof}

 The inequality \eqref{4.2} is implied by the usual Riemann Hypothesis for $\Lambda(f,s)$.   Note that RH for $\Lambda(f,s)$ implies also 
 that the derviatives $\Lambda^{(j)}(f,s)$ satisfy RH.   Moreover, at the central point one sees that $\Lambda^{(j)}(f,3)=0$ for all 
 odd $j$, and that $\Lambda^{(j)}(f,3) \ge 0$ for all even $j$.  Therefore, taking Taylor expansions around $3$, we see that 
 $$ 
 \Lambda(f,5) + 3\Lambda(f,3) = 4 \Lambda(f,3) + \sum_{j=1}^{\infty} \frac{\Lambda^{(2j)}(f,3)}{(2j)!} 2^{2j} 
 \ge 4 \Lambda(f,3) + 4 \sum_{j=1}^{\infty} \frac{\Lambda^{(2j)}(f,3)}{(2j)!} = 4\Lambda(f,4). 
 $$ 
 This reasoning in general explains why the period polynomial has the right sign at $\pi$ (see \eqref{1.7}). 
 
\section{Weights between $8$ and $14$:  Applications of results of P{\' o}lya and Szeg{\" o}} 
 
\noindent Classical work of P{\' o}lya \cite{Polya} and Szeg{\" o} \cite{Szego} considers trigonometric polynomials 
 \begin{displaymath}
 \begin{split}
 u(\theta) &= a_0 + a_1 \cos \theta +a_2 \cos (2\theta) + \ldots +a_n \cos (n\theta), \\
 v(\theta) &=   a_1 \sin \theta + a_2 \sin (2\theta) + \ldots + a_n \sin (n\theta).
 \end{split}
 \end{displaymath}
 If $0\le a_0 \le a_1 \leq a_2 \ldots \le a_{n-1} < a_n$, then Szeg{\" o} showed that $u$ and $v$ both 
 have exactly $n$ zeros in $[0,\pi)$ and that these zeros are simple.  Each interval $(\frac{\ell-\frac 12}{n+\frac 12} \pi, \frac{\ell+\frac 12}{n+\frac 12 }\pi)$ 
 for $\ell =1$, $\ldots$, $n$ has precisely one zero of $u$, and apart from $\theta=0$, each interval $(\frac{\ell}{n+\frac 12}\pi, \frac{\ell+1}{{n+\frac 12}} \pi)$ 
 for $1\le \ell\le n-1$ has exactly one zero of $v$.  His proof is a simple sign change argument using the positivity of the 
 Fej\'er kernel.  
 
 When the level is suitably large, these results apply and provide a quick 
 proof of Theorem~\ref{NHecke}.
 For weight $k$, for Szeg{\" o}'s 
 theorem to apply we must verify the criteria 
 \begin{equation} 
 \label{5.1} 
 \binom{2m}{m} \Lambda(f,\tfrac k2 ) \le 2 \binom{2m}{m+1} \Lambda(f,\tfrac k2 +1), 
 \end{equation}
 and for all $1\le j\le m-1$ that 
 \begin{equation} 
 \label{5.2} 
 \binom{2m}{m+j} \Lambda(f,\tfrac k2 +j) \le \binom{2m}{m+j+1} \Lambda(f,\tfrac k2 + j+1).
 \end{equation} 
 
 Since $\Lambda(f,\tfrac k2) \le \Lambda(f,\tfrac k2 +1)$, the condition \eqref{5.1} is immediate for all $k\ge 4$.  
 Now suppose $k\ge 6$.  Using the definition of $\Lambda$, and simplifying a little, the condition \eqref{5.2} 
 becomes (for $1\le j\le m-1$) 
 $$
 \sqrt{N} \ge \frac{2\pi}{(k/2-j-1)} \frac{L(f,\tfrac k2 +j)}{L(f,\tfrac k2 +j+1)}, 
 $$
 and by Lemma \ref{lem2.2} we conclude that our criterion \eqref{5.2} is met if 
 \begin{equation} 
 \label{5.3} 
 N \ge \max_{1\le j\le k/2-2} \Big(\frac{2\pi}{k/2-j-1} \Big)^2 \frac{\zeta(j+1/2)^4}{\zeta(j+3/2)^4}. 
 \end{equation}  
 
 For any given $k$, we can compute the bound \eqref{5.3}.  Thus, for 
 $k=8$, it suffices to take $N\ge 142$; for $k=10$ it suffices to have $N\ge 64$; for $k=12$ it suffices to have 
 $N\ge 45$;  for $k=14$ it suffices to have $N\ge 42$.   We have used {\text {\tt sage}} to check (\ref{5.2}) for those newforms not covered by (\ref{5.3}) for weights $8\leq k\leq 14$. The zeros of those newforms which do not satisfy (\ref{5.2}) still lie on $|z|=1/\sqrt{N}$.

\begin{remark}
Eventually, this cannot furnish a bound better than $4\pi^2$ for $N$, and so we must turn to another approach for large $k$ and small $N,$ which we carry out in the next section.  
 \end{remark}

  \section{Larger weights:  A second approach}  
  
  \noindent Here we consider larger weights by reformulating the previous approach of \cite{CFI} and \cite{ER}.
  Recast the definition \eqref{1.5} of $P_f(z)$ as 
  $$
  P_f(z) = (2m)! \Big(\frac{\sqrt{N}}{2\pi}\Big)^{2m+1} L(f,2m+1) Q_f(z),
  $$ 
  where 
  \begin{equation} 
  \label{6.1} 
  Q_f(z) =z^m \sum_{j=0}^{m-1} \frac{1}{j!} 
  \Big(\frac{2\pi }{z\sqrt{N}}\Big)^j \frac{L(f,2m+1-j)}{L(f,2m+1)} + \frac{1}{2(m!)^2} \Big(\frac{2\pi}{\sqrt{N}}\Big)^{2m+1} \frac{\Lambda(f,\frac k2)}{L(f,2m+1)}. 
  \end{equation}
  We wish to show that on the unit circle $|z|=1$, the real and imaginary parts of $Q_f(z)$ (which correspond to the even and odd signs of the functional equation) have exactly $2m$ zeros.  
  
  Now let us write 
  $$ 
  Q_f(z) = z^m \exp\Big( \frac{2\pi}{z\sqrt{N}}\Big) + S_1(z) + S_2(z) + S_3(z), 
  $$ 
  with 
  $$ 
  S_1(z) = z^m \sum_{j=1}^{m-1} \frac{1}{j!} \Big(\frac{2\pi}{z\sqrt{N}}\Big)^j \Big( \frac{L(f,2m+1-j)}{L(f,2m+1)}-1\Big),
  $$
  $$
  S_2(z) = - z^m\sum_{j=m}^{\infty}  \frac{1}{j!} \Big(\frac{2\pi}{z\sqrt{N}}\Big)^j, \qquad \text{and} \qquad S_3(z) = 
  \frac{1}{2(m!)^2} \Big(\frac{2\pi}{\sqrt{N}}\Big)^{2m+1} \frac{\Lambda(f,\tfrac k2)}{L(f,2m+1)}. 
  $$
  
  For $z=e^{i\theta}$ on the unit circle, the argument of $z^m \exp(2\pi/(z\sqrt{N}))$ is 
  $m\theta -2\pi (\sin \theta)/\sqrt{N},$ which is monotone increasing as $\theta$ varies from $0$ to $2\pi$, and changes by $2\pi m$ overall.  Therefore the real and imaginary parts of $z^m \exp(2\pi/(z\sqrt{N}))$ both have exactly $2m$ zeros.   More precisely, consider first the real part of $z^m \exp(2\pi/(z\sqrt{N})) 
  = \cos(m\theta-2\pi (\sin \theta)/\sqrt{N}) \exp(2\pi (\cos\theta)/\sqrt{N})$, and clearly we can find $m$ values of 
  $\theta$ with $\cos(m\theta-2\pi (\sin \theta)/\sqrt{N}) =1$ and $m$ interlacing values where it is $-1$.  Between 
  two such interlacing values there must be a zero of the real part.  Further, since $\exp(2\pi (\cos \theta)/\sqrt{N}) \ge 
  \exp(-2\pi/\sqrt{N})$ for all $\theta$, if 
  \begin{equation} 
  \label{6.2} 
  |S_1(z) + S_2(z) + S_3(z)| < \exp\Big( -\frac{2\pi}{\sqrt{N}}\Big),  
 \end{equation} 
 then the real part of $Q_f(z)$ will also have sign changes and thus a zero in these intervals.  A similar argument applies to the imaginary part of $Q_f(z)$, and so it suffices to check the criterion \eqref{6.2}.  
 
  Now by Lemma \ref{lem2.2} we see that $L(f,2m+1-j)/L(f,2m+1)-1 \le \zeta(\frac 12+m-j)^2-1$ so that 
  $$ 
  |S_1(z)+S_2(z)| \le \sum_{j=1}^{m-1} \frac{1}{j!} \Big(\frac{2\pi}{\sqrt{N}}\Big)^{j} (\zeta(\tfrac 12+m-j)^2-1)+ \sum_{j=m}^{\infty} 
  \frac{1}{j!} \Big(\frac{2\pi}{\sqrt{N}}\Big)^j. 
  $$ 
  For the term $j=m-1$, note that $\zeta(\frac 32)^2-1 \le \frac{35}{6}$ by direct computation.  Note that for $2^x (\zeta(\frac 12+x)^2-1)$ is decreasing in $x\ge 2$ and so may be bounded by $4(\zeta(5/2)^2-1) \le \frac{16}{5}$.  Using this observation for smaller values of $j$, we obtain 
  $$ 
  |S_1(z)+S_2(z)| \le \frac{16}{5} \sum_{j=1}^{m-1} \frac{1}{j!}   \Big(\frac{2\pi}{\sqrt{N}}\Big)^{j} \frac{2^j}{2^m} 
  + \frac{17}{4} \frac{1}{(m-1)!} \Big(\frac{2\pi}{\sqrt{N}}\Big)^{m-1} + \sum_{j=m}^{\infty} \frac{1}{j!} 
 \Big(\frac{2\pi}{\sqrt{N}}\Big)^{j} \frac{2^j}{2^m}. 
 $$
 Combining the first and third terms, we conclude that 
  \begin{equation} 
  \label{6.3}
  |S_1(z) + S_2(z)| \le \frac{16}{5} 2^{-m} \Big(\exp\Big(\frac{4\pi}{\sqrt{N}}\Big)-1\Big)  + 
  \frac{17}{4} \frac{1}{(m-1)!} \Big(\frac{2\pi}{\sqrt{N}}\Big)^{m-1}. 
  \end{equation} 
  
  To bound $S_3(z)$, note that 
  $\Lambda(f,\frac k2) \le \Lambda(f,\frac{k}{2}+1) \le (\frac{\sqrt{N}}{2\pi})^{m+2}(m+1)!\zeta(\frac 32)^2,$ 
  so  for $m\ge 7$ we have 
  $$ 
  |S_3(z)| \le \frac{m+1}{2(m!)} \Big( \frac{2\pi}{\sqrt{N}}\Big)^{m-1} \zeta(\tfrac 32)^2 \le \frac{4}{(m-1)!} \Big(\frac{2\pi}{\sqrt{N}}\Big)^{m-1}. 
   $$
Combining this with \eqref{6.3}, we conclude that 
\begin{equation} 
\label{6.4} 
|S_1(z)|+|S_2(z)|+|S_3(z)| \le \frac{16}{5} \frac{1}{2^m} \Big(\exp\Big(\frac{4\pi}{\sqrt{N}}\Big)-1\Big) + \frac{33}{4} \frac{1}{(m-1)!} 
\Big(\frac{2\pi}{\sqrt{N}}\Big)^{m-1}. 
\end{equation} 

Thus, to verify the condition \eqref{6.2}, we need only ensure that 
\begin{equation} 
\label{6.5}
\frac{16}{5} \frac{1}{2^m} \Big(\exp\Big(\frac{4\pi}{\sqrt{N}}\Big) -1\Big)+ \frac{33}{4} \frac{1}{(m-1)!} 
\Big(\frac{2\pi}{\sqrt{N}}\Big)^{m-1} < \exp\Big(-\frac{2\pi}{\sqrt{N}}\Big). 
\end{equation} 
For values of $m$ at least as large as the figure in the first row, the table below gives 
a bound $N(m)$ such that estimate \eqref{6.5} holds for all $N\ge N(m)$: 

\begin{center}
\begin{tabular} { |c || c | c |c|c|c|c|c|c|c|c|c|c|c|c|c|c|c|c|c|c|}
\hline
   $m$&  29& 21& 18& 16& 14 & 13& 12& 11& 10 & 9&8&7 \\
\hline
  $N(m)$& 1&  2&  3&  4&  5& 6& 7& 9& 11&14& 20& 28\\
  \hline
  \end{tabular}
  \end{center} 
  We  used {\text {\tt sage}} to confirm Theorem~\ref{NHecke} for the finitely many newforms missed by
  (\ref{6.5}).

\section{Proof of Theorem \ref{equi}.}\label{Equidistribution} 

\noindent The weight $4$ case was already treated in Section 3.  
For $m\ge 2$ (that is weights $k\ge 6$), the argument in Section 6 shows that for $z=e^{i\theta}$ on 
the unit circle we have 
$$ 
Q_f(z) = \exp\Big(im\theta + \frac{2\pi}{\sqrt{N}} e^{-i\theta}\Big) + O\Big( \frac{1}{2^m \sqrt{N}}\Big). 
$$ 
Thus we have that
$$ 
\text{Re}(Q_f(z)) = \exp\Big(\frac{2\pi}{\sqrt{N}} \cos\theta\Big) \cos\Big( m\theta -\frac{2\pi}{\sqrt{N}} \sin\theta\Big) + O\Big(\frac{1}{2^m\sqrt{N}}\Big).
$$ 
For $\theta \in [0,2\pi)$ the first term above vanishes when $m\theta -2\pi (\sin \theta)/\sqrt{N} = \frac{\pi}{2} + \ell \pi$ with $0\le \ell \le 2m-1$.  For such a point $\theta_\ell$, if we consider the values at $\theta_\ell - C/(2^m\sqrt{N})$ and $\theta_\ell +C/(2^m \sqrt{N})$ for a suitable constant $C>0$ (and if $2^m \sqrt{N}$ is large enough) then 
Re$(Q_f(z))$ has differing signs at these points, and hence a zero in between.  When $\epsilon(f)=1$, the zeros of the period polynomial $r_f(z)$ are located at $1/(i\sqrt{N})$ times the zeros of Re$(Q_f(z))$, and this proves Theorem \ref{equi} in this case.  The case when $\epsilon(f)=-1$ corresponds to Im$(Q_f(z)),$ and a similar argument applies here.

\section{Remarks on the calculations}

\noindent In the previous sections we  proved Theorem~\ref{NHecke} for $k=4, 6$ and $k\geq 42$.
For $8\leq k\leq 40$ finitely many newforms remain to complete the proof 
(see the discussions after (5.3) and (6.5)). 
We used inequality (5.3) for $8\leq k\leq 14$.
The most levels remain for weight $k=8$; we are left to consider those newforms with $N\leq 141$. 
For weights $16\leq k\leq 40$ we employed (6.5). The table after (6.5) gives the remaining levels.  The most levels  remain for weight $k=16$; we
are left with $N\leq 27$.

Using {\text {\tt sage}}
we confirmed Theorem~\ref{NHecke} for these remaining newforms. Running the commands {\text {\tt CuspForms}} and {\text {\tt newforms}} on a laptop,
we had no difficulty computing these newforms.
We then used
Dokchitser's {\tt sage} $L$-functions calculator to
compute the values $\Lambda(f,1),\dots, \Lambda(f,k-1)$ to very high precision. We tested inequality (5.2), and found that it held for many of the remaining newforms. However, (5.2) fails for some newforms with low weight and level.  For example, (5.2) fails for some weight $k=8$ newforms with $N\in \{2, 3, 5-17, 19\}$.

For the forms which do not satisfy (5.2), we  computed
the trigonometric polynomials and checked that on the unit disk that they have the required
number of sign changes for the truth of Theorem~\ref{NHecke}. 
As an example, consider the unique newform $f\in S_{10}(\Gamma_0(12)).$
We have that
\begin{eqnarray*}
&&L(f,1)\approx 343.041936898889, L(f,2)\approx 140.422365373567, 
L(f,3)\approx 32.9164131544840, \\
&&L(f,4)\approx 6.41626479306637, 
L(f,5)\approx 1.71889934464323, \ldots,
\end{eqnarray*}
which in turn implies for $z=e^{i\theta}$ that
\begin{displaymath}
\begin{split}
(P_{f}(z)+&\epsilon(f) P_{f}(1/z))/2
\approx 189.128932153817\cos(4\theta) + 341.466246468159\cos(3\theta) \\
&+308.910589184567\cos(2\theta) + 199.188643773093\cos(\theta) +73.5501402820398.
\end{split}
\end{displaymath}
This has four zeros for $\theta \in [0,\pi)$ as required, and they are in the intervals
\begin{eqnarray*}
\left(\frac{4\pi}{20},\frac{5\pi}{20}\right),
\left(\frac{10\pi}{20},\frac{11\pi}{20}\right),
\left(\frac{14\pi}{20},\frac{15\pi}{20}\right),
\left(\frac{18\pi}{20},\frac{19\pi}{20}\right).
\end{eqnarray*}

\end{document}